\documentclass[oneside, 12pt]{amsart}
\pdfoutput=1
\makeatletter
\let\origsection=\section \def\section{\@ifstar{\origsection*}{\mysection}}
\def\mysection{\@startsection{section}{1}\z@{.7\linespacing\@plus\linespacing}{.5\linespacing}{\normalfont\scshape\centering\S}}
\makeatother

\usepackage{amsmath,amssymb,amsthm}
\usepackage{mathrsfs}
\usepackage{mathabx}\changenotsign
\usepackage{dsfont}
\usepackage{bbm}

\usepackage{xcolor}
\usepackage[backref=section]{hyperref}
\usepackage[ocgcolorlinks]{ocgx2}
\hypersetup{
	colorlinks=true,
	linkcolor={red!60!black},
	citecolor={green!60!black},
	urlcolor={blue!60!black},
}

\usepackage[open,openlevel=2,atend]{bookmark}

\usepackage[abbrev,msc-links,backrefs]{amsrefs}
\usepackage{doi}

\renewcommand{\PrintDOI}[1]{\doi{#1}}

\usepackage[T1]{fontenc}
\usepackage{lmodern}
\usepackage[babel]{microtype}
\usepackage[english]{babel}

\usepackage{geometry}
\numberwithin{equation}{section}
\numberwithin{figure}{section}

\usepackage{enumitem}

\let\polishlcross=\l
\def\l{\ifmmode\ell\else\polishlcross\fi}

\let\emptyset=\varnothing
\let\setminus=\smallsetminus

\makeatletter
\def\moverlay{\mathpalette\mov@rlay}
\def\mov@rlay#1#2{\leavevmode\vtop{   \baselineskip\z@skip \lineskiplimit-\maxdimen
		\ialign{\hfil$\m@th#1##$\hfil\cr#2\crcr}}}
\newcommand{\charfusion}[3][\mathord]{
	#1{\ifx#1\mathop\vphantom{#2}\fi
		\mathpalette\mov@rlay{#2\cr#3}
	}
	\ifx#1\mathop\expandafter\displaylimits\fi}
\makeatother

\newcommand{\dcup}{\charfusion[\mathbin]{\cup}{\cdot}}

\DeclareFontFamily{U}  {MnSymbolC}{}
\DeclareSymbolFont{MnSyC}         {U}  {MnSymbolC}{m}{n}
\DeclareFontShape{U}{MnSymbolC}{m}{n}{
	<-6>  MnSymbolC5
	<6-7>  MnSymbolC6
	<7-8>  MnSymbolC7
	<8-9>  MnSymbolC8
	<9-10> MnSymbolC9
	<10-12> MnSymbolC10
	<12->   MnSymbolC12}{}
\DeclareMathSymbol{\powerset}{\mathord}{MnSyC}{180}

\usepackage{tikz}
\usetikzlibrary{calc,decorations.pathmorphing}
\usetikzlibrary{arrows,decorations.pathreplacing}
\pgfdeclarelayer{background}
\pgfdeclarelayer{foreground}
\pgfdeclarelayer{front}
\pgfsetlayers{background,main,foreground,front}

\usepackage{multicol}
\usepackage{subcaption}
\newcommand{\pedge}[9]{
	
	\ifx\relax#6\relax
	\def\qoffs{0pt}
	\else
	\def\qoffs{#6}
	\fi
	
	\def\phedge{
		($#1+#5!\qoffs!-90:#2-#5$) -- 
		($#2+#1!\qoffs!-90:#3-#1$) -- 
		($#3+#2!\qoffs!-90:#4-#2$) -- 
		($#4+#3!\qoffs!-90:#5-#3$) -- 
		($#5+#4!\qoffs!-90:#1-#4$) -- cycle}

	\coordinate (12) at ($#1!\qoffs!90:#2$);
	\coordinate (15) at ($#1!\qoffs!-90:#5$);
	\coordinate (23) at ($#2!\qoffs!90:#3$);
	\coordinate (21) at ($#2!\qoffs!-90:#1$);
	\coordinate (34) at ($#3!\qoffs!90:#4$);
	\coordinate (32) at ($#3!\qoffs!-90:#2$);
	\coordinate (45) at ($#4!\qoffs!90:#5$);
	\coordinate (43) at ($#4!\qoffs!-90:#3$);
	\coordinate (51) at ($#5!\qoffs!90:#1$);
	\coordinate (54) at ($#5!\qoffs!-90:#4$);

	\def\nphedge{
		(15) let \p1=($(15)-#1$), \p2=($(12)-#1$) in 
		arc[start angle={atan2(\y1,\x1)}, delta angle={atan2(\y2,\x2)-atan2(\y1,\x1)-360*(atan2(\y2,\x2)-atan2(\y1,\x1)>0)}, x radius=\qoffs, y radius=\qoffs] --
		(21) let \p1=($(21)-#2$), \p2=($(23)-#2$) in 
		arc[start angle={atan2(\y1,\x1)}, delta angle={atan2(\y2,\x2)-atan2(\y1,\x1)-360*(atan2(\y2,\x2)-atan2(\y1,\x1)>0)}, x radius=\qoffs, y radius=\qoffs] --
		(32) let \p1=($(32)-#3$), \p2=($(34)-#3$) in 
		arc[start  angle={atan2(\y1,\x1)}, delta angle={atan2(\y2,\x2)-atan2(\y1,\x1)-360*(atan2(\y2,\x2)-atan2(\y1,\x1)>0)}, x radius=\qoffs, y radius=\qoffs] --
		(43) let \p1=($(43)-#4$), \p2=($(45)-#4$) in 
		arc[start angle={atan2(\y1,\x1)}, delta angle={atan2(\y2,\x2)-atan2(\y1,\x1)-360*(atan2(\y2,\x2)-atan2(\y1,\x1)>0)}, x radius=\qoffs, y radius=\qoffs] --
		(54) let \p1=($(54)-#5$), \p2=($(51)-#5$) in 
		arc[start angle={atan2(\y1,\x1)}, delta angle={atan2(\y2,\x2)-atan2(\y1,\x1)-360*(atan2(\y2,\x2)-atan2(\y1,\x1)>0)}, x radius=\qoffs, y radius=\qoffs] --
		cycle}

	\ifx\relax#7\relax
	\def\plwidth{1pt}
	\else
	\def\plwidth{#7}
	\fi
	
	\ifx\relax#9\relax
	\fill \nphedge;
	\else
	\fill[#9]\nphedge;
	\fi
	
	\ifx\relax#8\relax
	\draw[line width=\plwidth,rounded corners=\qoffs]\nphedge;
	\else
	\draw[line width=\plwidth,#8]\nphedge;
	\fi
}

\newcommand{\qedge}[7]{
	
	\ifx\relax#4\relax
	\def\qoffs{0pt}
	\else
	\def\qoffs{#4}
	\fi
	
	\def\qhedge{
		($#1+#3!\qoffs!-90:#2-#3$) --
		($#2+#1!\qoffs!-90:#3-#1$) --
		($#3+#2!\qoffs!-90:#1-#2$) -- cycle}

	\coordinate (12) at ($#1!\qoffs!90:#2$);
	\coordinate (13) at ($#1!\qoffs!-90:#3$);
	\coordinate (23) at ($#2!\qoffs!90:#3$);
	\coordinate (21) at ($#2!\qoffs!-90:#1$);
	\coordinate (31) at ($#3!\qoffs!90:#1$);
	\coordinate (32) at ($#3!\qoffs!-90:#2$);
	
	\def\nqhedge{
		(13) let \p1=($(13)-#1$), \p2=($(12)-#1$) in
		arc[start angle={atan2(\y1,\x1)}, delta angle={atan2(\y2,\x2)-atan2(\y1,\x1)-360*(atan2(\y2,\x2)-atan2(\y1,\x1)>0)}, x radius=\qoffs, y radius=\qoffs] --
		(21) let \p1=($(21)-#2$), \p2=($(23)-#2$) in
		arc[start angle={atan2(\y1,\x1)}, delta angle={atan2(\y2,\x2)-atan2(\y1,\x1)-360*(atan2(\y2,\x2)-atan2(\y1,\x1)>0)}, x radius=\qoffs, y radius=\qoffs] --
		(32) let \p1=($(32)-#3$), \p2=($(31)-#3$) in
		arc[start angle={atan2(\y1,\x1)}, delta angle={atan2(\y2,\x2)-atan2(\y1,\x1)-360*(atan2(\y2,\x2)-atan2(\y1,\x1)>0)}, x radius=\qoffs, y radius=\qoffs] --
		cycle}
	
	\ifx\relax#5\relax
	\def\qlwidth{1pt}
	\else
	\def\qlwidth{#5}
	\fi
	
	\ifx\relax#7\relax
	\fill \nqhedge;
	\else
	\fill[#7]\nqhedge;
	\fi
	
	\ifx\relax#6\relax
	\draw[line width=\qlwidth,rounded corners=\qoffs]\nqhedge;
	\else
	\draw[line width=\qlwidth,#6]\nqhedge;
	\fi
}

\newcommand{\redge}[8]{
	
	\ifx\relax#5\relax
	\def\qoffs{0pt}
	\else
	\def\qoffs{#5}
	\fi
	
	\def\rhedge{
		($#1+#4!\qoffs!-90:#2-#4$) -- 
		($#2+#1!\qoffs!-90:#3-#1$) -- 
		($#3+#2!\qoffs!-90:#4-#2$) -- 
		($#4+#3!\qoffs!-90:#1-#3$) -- cycle}

	\coordinate (12) at ($#1!\qoffs!90:#2$);
	\coordinate (14) at ($#1!\qoffs!-90:#4$);
	\coordinate (23) at ($#2!\qoffs!90:#3$);
	\coordinate (21) at ($#2!\qoffs!-90:#1$);
	\coordinate (34) at ($#3!\qoffs!90:#4$);
	\coordinate (32) at ($#3!\qoffs!-90:#2$);
	\coordinate (41) at ($#4!\qoffs!90:#1$);
	\coordinate (43) at ($#4!\qoffs!-90:#3$);
	
	\def\nrhedge{
		(14) let \p1=($(14)-#1$), \p2=($(12)-#1$) in 
		arc[start angle={atan2(\y1,\x1)}, delta angle={atan2(\y2,\x2)-atan2(\y1,\x1)-360*(atan2(\y2,\x2)-atan2(\y1,\x1)>0)}, x radius=\qoffs, y radius=\qoffs] --
		(21) let \p1=($(21)-#2$), \p2=($(23)-#2$) in 
		arc[start angle={atan2(\y1,\x1)}, delta angle={atan2(\y2,\x2)-atan2(\y1,\x1)-360*(atan2(\y2,\x2)-atan2(\y1,\x1)>0)}, x radius=\qoffs, y radius=\qoffs] --
		(32) let \p1=($(32)-#3$), \p2=($(34)-#3$) in 
		arc[start angle={atan2(\y1,\x1)}, delta angle={atan2(\y2,\x2)-atan2(\y1,\x1)-360*(atan2(\y2,\x2)-atan2(\y1,\x1)>0)}, x radius=\qoffs, y radius=\qoffs] --
		(43) let \p1=($(43)-#4$), \p2=($(41)-#4$) in 
		arc[start angle={atan2(\y1,\x1)}, delta angle={atan2(\y2,\x2)-atan2(\y1,\x1)-360*(atan2(\y2,\x2)-atan2(\y1,\x1)>0)}, x radius=\qoffs, y radius=\qoffs] --
		cycle}
	
	\ifx\relax#6\relax
	\def\rlwidth{1pt}
	\else
	\def\rlwidth{#6}
	\fi
	
	\ifx\relax#8\relax
	\fill \nrhedge;
	\else
	\fill[#8]\nrhedge;
	\fi
	
	\ifx\relax#7\relax
	\draw[line width=\rlwidth,rounded corners=\qoffs]\nrhedge;
	\else
	\draw[line width=\rlwidth,#7]\nrhedge;
	\fi
}

\let\epsilon=\varepsilon
\let\eps=\epsilon
\let\rho=\varrho
\let\theta=\vartheta

\newcommand{\cC}{\mathcal{C}}
\newcommand{\cE}{\mathcal{E}}

\newcommand{\cT}{\mathcal{T}}

\newtheoremstyle{note}  {4pt}  {4pt}  {\sl}  {}  {\bfseries}  {.}  {.5em}          {}
\newtheoremstyle{introthms}  {3pt}  {3pt}  {\itshape}  {}  {\bfseries}  {.}  {.5em}          {\thmnote{#3}}
\newtheoremstyle{remark}  {2pt}  {2pt}  {\rm}  {}  {\bfseries}  {.}  {.3em}          {}

\theoremstyle{plain}
\newtheorem{theorem}{Theorem}[section]
\newtheorem{lemma}[theorem]{Lemma}

\theoremstyle{note}

\theoremstyle{remark}

\newtheorem{observation}[theorem]{Observation}

\newtheorem{proposition}[theorem]{Proposition}

\usepackage{lineno}
\newcommand*\patchAmsMathEnvironmentForLineno[1]{
	\expandafter\let\csname old#1\expandafter\endcsname\csname #1\endcsname
	\expandafter\let\csname oldend#1\expandafter\endcsname\csname end#1\endcsname
	\renewenvironment{#1}
	{\linenomath\csname old#1\endcsname}
	{\csname oldend#1\endcsname\endlinenomath}}
\newcommand*\patchBothAmsMathEnvironmentsForLineno[1]{
	\patchAmsMathEnvironmentForLineno{#1}
	\patchAmsMathEnvironmentForLineno{#1*}}
\AtBeginDocument{
	\patchBothAmsMathEnvironmentsForLineno{equation}
	\patchBothAmsMathEnvironmentsForLineno{align}
	\patchBothAmsMathEnvironmentsForLineno{flalign}
	\patchBothAmsMathEnvironmentsForLineno{alignat}
	\patchBothAmsMathEnvironmentsForLineno{gather}
	\patchBothAmsMathEnvironmentsForLineno{multline}
}

\usepackage{scalerel}

\newsavebox\vdegbox
\savebox\vdegbox{\tikz{
		\draw[black,fill=black] (90:1) circle (.35);
		\draw[black,line width=0.10cm] (210:1) circle (.30);
		\draw[black,line width=0.10cm] (330:1) circle (.30);
		\draw[opacity=0] (0:1.2) circle (0.1);
}}

\newsavebox\vvbox
\savebox\vvbox{\tikz{
		\draw[black,line width=0.10cm] (90:1) circle (.30);
		\draw[black,fill=black] (210:1) circle (.35);
		\draw[black,fill=black] (330:1) circle (.35);
		\draw[opacity=0] (0:1.2) circle (0.1);
}}

\newsavebox\pdegbox
\savebox\pdegbox{\tikz{
		\draw[black,line width=0.10cm] (90:1) circle (.30);
		\draw[black,fill=black] (210:1) circle (.35);
		\draw[black,fill=black] (330:1) circle (.35);
		\draw[black,line width=0.28cm ] (210:1) -- (330:1);
		\draw[opacity=0] (0:1.2) circle (0.1);
}}

\newsavebox\vvvbox
\savebox\vvvbox{\tikz{
		\draw[black,fill=black] (90:1) circle (.35);
		\draw[black,fill=black] (210:1) circle (.35);
		\draw[black,fill=black] (330:1) circle (.35);
		\draw[opacity=0] (0:1.2) circle (0.1);
}}

\newsavebox\evbox
\savebox\evbox{\tikz{
		\draw[black,fill=black] (90:1) circle (.35);
		\draw[black,fill=black] (210:1) circle (.35);
		\draw[black,fill=black] (330:1) circle (.35);
		\draw[black,line width=0.28cm ] (210:1) -- (330:1);
		\draw[opacity=0] (0:1.2) circle (0.1);
}}

\newsavebox\eebox
\savebox\eebox{\tikz{
		\draw[black,fill=black] (90:1) circle (.35);
		\draw[black,fill=black] (210:1) circle (.35);
		\draw[black,fill=black] (330:1) circle (.35);
		\draw[black,line width=0.28cm ] (90:1) -- (330:1);
		\draw[black,line width=0.28cm ] (90:1) -- (210:1);
		\draw[opacity=0] (0:1.2) circle (0.1);
}}

\newsavebox\eeebox
\savebox\eeebox{\tikz{
		\draw[black,fill=black] (90:1) circle (.35);
		\draw[black,fill=black] (210:1) circle (.35);
		\draw[black,fill=black] (330:1) circle (.35);
		\draw[black,line width=0.28cm ] (90:1) -- (330:1);
		\draw[black,line width=0.28cm ] (90:1) -- (210:1);
		\draw[black,line width=0.28cm ] (210:1) -- (330:1);
		\draw[opacity=0] (0:1.2) circle (0.1);
}}

\makeatletter
\newcommand{\overrighharpoonup}[1]{\ThisStyle{%
		\vbox {\m@th\ialign{##\crcr
				\rightharpoonupfill \crcr
				\noalign{\kern-\p@\nointerlineskip}
				$\hfil\SavedStyle#1\hfil$\crcr}}}}

\def\rightharpoonupfill{%
	$\SavedStyle\m@th\mkern+0.8mu\cleaders\hbox{$\shortbar\mkern-4mu$}\hfill\rightharpoonuptip\mkern+0.8mu$}

\def\rightharpoonuptip{%
	\raisebox{\z@}[2pt][1pt]{\scalebox{0.55}{$\SavedStyle\rightharpoonup$}}}

\def\shortbar{%
	\smash{\scalebox{0.55}{$\SavedStyle\relbar$}}}
\makeatother

\makeatletter
\newcommand{\overlefharpoonup}[1]{\ThisStyle{%
		\vbox {\m@th\ialign{##\crcr
				\leftharpoonupfill \crcr
				\noalign{\kern-\p@\nointerlineskip}
				$\hfil\SavedStyle#1\hfil$\crcr}}}}

\def\leftharpoonupfill{%
	$\SavedStyle\m@th\mkern+0.8mu\cleaders\hbox{$\shortbar\mkern-4mu$}\hfill\leftharpoonuptip\mkern+0.8mu$}

\def\leftharpoonuptip{%
	\raisebox{\z@}[2pt][1pt]{\scalebox{0.55}{$\SavedStyle\leftharpoonup$}}}

\makeatletter
\newsavebox\myboxA
\newsavebox\myboxB
\newlength\mylenA

\newcommand*\xoverline[2][0.75]{%
	\sbox{\myboxA}{$\m@th#2$}%
	\setbox\myboxB\null% Phantom box
	\ht\myboxB=\ht\myboxA%
	\dp\myboxB=\dp\myboxA%
	\wd\myboxB=#1\wd\myboxA% Scale phantom
	\sbox\myboxB{$\m@th\overline{\copy\myboxB}$}%  Overlined phantom
	\setlength\mylenA{\the\wd\myboxA}%   calc width diff
	\addtolength\mylenA{-\the\wd\myboxB}%
	\ifdim\wd\myboxB<\wd\myboxA%
	\rlap{\hskip 0.5\mylenA\usebox\myboxB}{\usebox\myboxA}%
	\else
	\hskip -0.5\mylenA\rlap{\usebox\myboxA}{\hskip 0.5\mylenA\usebox\myboxB}%
	\fi}
\makeatother

\begin{document}
	
	\title[Bootstrap percolation of extension hypergraphs]{Bootstrap percolation of extension hypergraphs}
	
	\author[Weichan Liu]{Weichan Liu}
	\address{School of Mathematics, Shandong University, Jinan, China}
	\email{wcliu@sdu.edu.cn}
	
	\author[Bjarne Sch\"{u}lke]{Bjarne Sch\"{u}lke}
	\address{Extremal Combinatorics and Probability Group, Institute for Basic Science, Daejeon, South Korea}
	\email{schuelke@ibs.re.kr}
	
	\author[Xin Zhang]{Xin Zhang}
	\address{School of Mathematics and Statistics, Xidian University, Xi'an, China}
	\email{xzhang@xidian.edu.cn}

	\begin{abstract}
        For~$k$-graphs~$F$ and~$H_0$ the~$F$-bootstrap percolation process (or~$F$-process) starting with~$H_0$ is a sequence~$(H_i)_{i\geq0}$ of~$k$-graphs such that~$H_{i+1}$ is obtained from~$H_i$ by adding all those~$e\in V(H_0)^{(k)}\setminus E(H_i)$ as edges that complete a new copy of~$F$.
        The running time of this~$F$-process, denoted by~$M_F(H_0)$, is the smallest~$i$ with~$H_i=H_{i+1}$.
        Bollob\'as proposed the problem of determining the maximum running time for~$n\in\mathds{N}$, i.e., $$M_F(n)=\max_{\vert V(H_0)\vert=n}M_F(H_0)\,.$$
        
        Recently, Noel and Ranganathan initiated the study of this quantity for~$k$-graphs.
        In this work, we determine the asymptotics of~$M_F(n)$ for a large class of~$k$-graphs.
	    Given a graph~$G=(V,E)$, the~$k$-extension of~$G$ is a~$k$-graph~$F^{(k)}(G)$ obtained from~$G$ by enlarging each edge with a~$(k-2)$-set of new vertices.
        We show that for every graph~$G$ on~$t$ vertices and every~$k\geq 3$,~$M_{F^{(k)}(G)}(n)\leq C_{k,t}$ for some constant~$C_{k,t}$ depending only on~$t$ and~$k$.
	\end{abstract}
	
	\maketitle
	\section{Introduction}
	
	A \emph{$k$-uniform hypergraph} (or \emph{$k$-graph}) $H$ consists of a vertex set $V(H)$ and a set of edges~$E(H) \subseteq V(H)^{(k)}$, where~$V(H)^{(k)} = \{ S \subseteq V(H) : |S| = k \}$.
	Given $k$-graphs $F$ and~$H$, a copy of~$F$ in~$H$ is a subhypergraph of~$H$ that is isomorphic to~$F$.
    The number of copies of~$F$ in~$H$ is denoted by~$c(F,H)$.
    The \emph{$F$-bootstrap percolation process} (or \emph{$F$-process}) \emph{starting with}~$H$ is the sequence~$(H_i)_{i\geq0}$ of~$k$-graphs with~$H_i$ defined inductively as follows.
    Set~$H_0=H$ and given~$H_i$ let~$H_{i+1}$ be the~$k$-graph with vertex set~$V(H)$ and edge set
    \begin{align*}
		E(H_{i+1})=E(H_i)\cup\big\{e\in V(H)^{(k)}:\: c(F,H_i\cup e)>c(F,H_i)\big\}\,,
	\end{align*}
    where~$H_i\cup e$ denotes the~$k$-graph~$\big(V(H),E(H_i)\cup\{e\}\big)$.

    This notion originates in the work of Bollob\'as~\cite{B:68}, where it arose in the context of weak saturation.
    Since then, connections to other areas such as to cellular automata, more general bootstrap-type dynamics~\cite{M:17}, and statistical physics (see~\cite{BBM:12}~and~\cite{CLR:79}) have been established.
    Bollob\'as~\cite{BPRS:17} further suggested the problem of determining the maximum running time of such processes (see also~\cite{GKP:17} for a random variant).
    More precisely, the running time of the~$F$-process starting with~$H$ is~$\tau_F(H):=\min\{t\in\mathbb{N}_{\geq0}:\:H_t = H_{t + 1}\}$.
    Given~$n\in\mathds{N}$ (and for~$k=2$), Bollob\'as asked to determine
	\[
	M_F(n) = \max \big\{ \tau_F(H) : H \text{ is a } k\text{-graph with } n \text{ vertices} \big\}\,.
	\]
	
	One major theme in this area is the maximum running time of~$K_r$-processes, where~$K_r$ is the complete graph on~$r$ vertices. 
	Bollob\'as, Przykucki, Riordan, and Sahasrabudhe~\cite{BPRS:17} and, independently, Matzke~\cite{M:15} showed that $M_{K_4}(n)=n-3$, for all $n\ge 3$. 
	In~\cite{BPRS:17} the authors also showed that~$M_{K_r}(n)\ge n^{2-\lambda_r-o(1)}$ for~$r\ge 5$, where~$\lambda_r$ is some explicit constant with~$\lambda_r\rightarrow 0$ as~$r\rightarrow \infty$, and conjectured that for all~$r\ge 5$,~$M_{K_r}(n)=o(n^2)$. 
	Balogh, Kronenberg, Pokrovskiy, and Szab\'o~\cite{BKPS:19} disproved this conjecture, showing  that~$M_{K_r}(n)=\Omega(n^2)$ for all~$r\ge 6$. 
	For~$r=5$, they used Behrend's construction for arithmetic progressions to show that~$M_{K_5}(n)$ is asymptotically bigger than~$n^{2-\eps}$ for all~$\eps>0$.
    It remains an open problem whether~$M_{K_5}(n)$ is quadratic in~$n$ or not.
	
    Considering graphs~$F$ that are not complete, Fabian, Morris, and Szabó~\cite{FMS:23} showed that~$M_{C_k}(n)$ is of order~$\log_{k-1}(n)$, where~$C_k$ is the cycle of length~$k$.
    In~\cite{FMS:25}, the same authors showed that for every tree~$T$,~$M_T(n)$ is a constant depending only on~$\vert V(T)\vert$, while the exact function remains unknown.
	They also explore how certain properties of~$F$ such as the minimum degree and the connectivity can imply a relatively small running time. 
    For more background on the~$F$-process we refer the reader to the recent survey by Fabian, Morris, and Szab\'o~\cite{FMS:26}.
    
	Recently, the investigation of~$F$-processes and their maximum running times was initiated for hypergraphs.
    Noel and Ranganathan~\cite{NR:22} showed that~$M_{K^{(k)}_r}(n)=\Theta(n^k)$ for~$k\geq3$ and~$r\geq k+2$, where~$K^{(k)}_r$ is the complete~$k$-graph on~$r$ vertices.
    Answering a question by Noel and Ranganathan, Hartarsky, and Lichev \cite{HL:22}, and, independently, Espuny Díaz, Janzer, Kronenberg, and Lada \cite{EJKL:22} proved that also~$M_{K^{(k)}_{k+1}}(n)=\Theta(n^k)$.
    Note that these results state that for the running time of hypergraph cliques the trivial upper bound already provides the right order of magnitude.
	In contrast, Espuny Díaz, Janzer, Kronenberg, and Lada~\cite{EJKL:22} also determined~$M_{K^{(3)-}_4}(n)$ to be linear in~$n$, which requires a non-trivial upper bound.

    In~\cite{EJKL:22} the authors suggested to study the asymptotic running time of the bootstrap percolation process for other hypergraphs which are not complete.
    Here we consider extension hypergraphs, which were for instance considered by Mubayi~\cite{M:06} in the context of the hypergraph Tur\'an problem.
	The \emph{$k$-extension}~$F^{(k)}(G)$ of a graph~$G$ is the~$k$-graph defined as follows.
    The vertex set is given by $$V(F^{(k)}(G))=V(G)\dcup\{v^i_e:\:i\in[k-2],e\in E(G)\}$$ (where~$v_e^i\neq v_{e'}^j$ if~$e\neq e'$ or~$i\neq j$) and the edge set is $$E(F^{(k)}(G))=\big\{e\cup\{v_e^1,\dots,v_e^{k-2}\}:\:e\in E(G)\big\}\,.$$
    %Note that~$F^{(k)}(G)$ has~$\vert V(G)\vert+\vert E(G)\vert\cdot(k-2)$ vertices and~$\vert E(G)\vert$ edges.
    Our main result is that the maximum running time of the~$F^{(k)}(G)$-process is constant in~$n$.
	
	\begin{theorem}\label{thm:main result}
		Given an integer~$k\geq3$ and a graph~$G$, there is some~$C\in\mathds{N}$ such that for all~$n\in\mathds{N}$ we have~$M_{F^{(k)}(G)}(n)\leq C$.
	\end{theorem}

    We remark that the case~$k=3$ is the most interesting one and indeed the proof can be significantly shortened for~$k\geq4$ as we point out in the proof sketch below.
    Note that the constant~$C$ depends on the uniformity and on~$G$.
    It is not difficult to see that for certain classes of graphs, such as cycles and paths,~$M_{F^{(k)}(G)}(n)$ is asymptotically bounded by an absolute constant not depending on~$\vert V(G)\vert$ and~$k$.
    However, we observe that in general the maximum running time of~$F^{(k)}(G)$ cannot be bounded by an absolute constant independent of~$\vert V(G)\vert$.
    Hence, Theorem~\ref{thm:main result} is in a sense best possible.
    
	\begin{proposition}\label{prop:not absolute constant lower bound}
		Let $T=K_{1,t-1}$ be the star with~$t$ vertices.
        Then $M_{F^{(k)}(T)}(n)\ge t-1$ for all sufficiently large $n$.
	\end{proposition}

    In this paper, we do not attempt to optimise the constants and we omit any divisibility issues unless they are relevant.
    We collect some definitions in Section~\ref{sec:notation}, then we prove Theorem~\ref{thm:main result} in Section~\ref{sec:main proof} and Proposition~\ref{prop:not absolute constant lower bound} in Section~\ref{sec:prop}.

    \subsection*{Overview of the proof}
    We commence with the simple observation that every pair in the ``center'' of a copy of~$F^{(k)}(G)$ in some~$H_i$ will have almost full degree after at most two more steps (Observation~\ref{obs:almostfulldegree}).
    This is the basic motivation for the definition of \emph{essential pairs}, which are particular pairs appearing in the center of a copy of~$F^{(k)}(G)$.
    After many steps, there must also be many essential pairs (Observation~\ref{obs:1/3realstep} and Observation~\ref{obs:atmosttwiceessential}).
    A key step in the proof is that if there are many \emph{good} vertices, that is, vertices contained in many center pairs, then the process will stop soon afterwards (Lemma~\ref{lem:good+6}).
    This is combined with another key lemma that no vertex can be contained in many essential pairs (Lemma~\ref{lem:essential+bounded}).
    At this point, we remark that for~$k\geq4$ the proof is significantly simpler since it can be observed that the process will stop shortly after there is a large enough matching of essential pairs.\footnote{Since the write-up is not much longer, we deal with a general uniformity~$k\geq3$ throughout the proof. But the reader may simply think of~$k=3$.}
    Next, assuming that the process does not terminate sooner, we choose a large matching of essential pairs whose ``associated copies'' of~$F^{(k)}(G)$ do not interact too much (Lemma~\ref{lem:matchingessential}).
    Now we analyse the interactions between any two such copies.
    Either certain vertices can be exchanged between the centers of these copies, or there has to be a ``reason'' why such an exchange is not possible.
    This reason takes the form of certain sets being non-empty (Lemma~\ref{lem:cannotreplace}).
    Eventually, by analysing an auxiliary directed graph representing those ``reasons'', we obtain a contradiction if there are too many copies between which certain vertices cannot be exchanged.
    This entails that between many copies the exchange must indeed be possible, which in turn will yield many good vertices (Observation~\ref{obs:manyalternatives+good}), completing the proof.

	\section{Notation}\label{sec:notation}
    Given integers~$i$ and~$j$, we write~$[i,j]=\{i,i+1,\dots,j-1,j\}$, where~$[i,j]=\{i\}$ if~$i=j$ and~$[i,j]=\emptyset$ if~$i>j$, and we set~$[i]=[1,i]$.
    For graphs and hypergraphs we often omit parentheses around sets of vertices, for instance, edges.
    Similarly, when considering a~$k$-graph~$H$, vertices~$v,w\in V(H)$ and an~$i$-set~$\mathbf{x}\in V(H)^{(i)}$, we write~$v\mathbf{x}=\mathbf{x}v=\{v\}\cup\mathbf{x}$ and~$vw\mathbf{x}=v\mathbf{x}w=\{v,w\}\cup\mathbf{x}$ etc.
    
    When~$G$ is a graph and we consider copies of~$F^{(k)}(G)$ in a~$k$-graph~$H$, we arbitrarily fix an isomorphism~$\varphi_F$ from each copy~$F$ to~$F^{(k)}(G)$.
    For a copy~$F$ of~$F^{(k)}(G)$ the \emph{center} of~$F$ is the graph whose vertex set is the preimage of~$V(G)$ under~$\varphi_F$, that is,~$\varphi_F^{-1}(V(G))$, and whose edge set is~$\big\{\varphi_F^{-1}(e):\:e\in E(G)\big\}$.
    We denote the center of~$F$ by~$C(F)$.
    The image of~$v\in V(F)$ under the fixed isomorphism is also referred to as \emph{the position of~$v$} (in~$F$).
    Given two copies~$F$ and~$F'$ of~$F^{(k)}(G)$, we say that~$v\in V(F)$ and~$v'\in V(F)$ \emph{play the same r\^ole} (in~$F$ and~$F'$, respectively) if their positions are the same.

    When dealing with centers of copies of~$F^{(k)}(G)$ in~$H$, we sometimes consider graphs that arise from a center by replacing one of its vertices by another vertex of~$H$.
    For this we introduce the following notation.
    Given an (ambient) vertex set~$V$, a graph~$C$ with~$V(C)\subseteq V$, as well as vertices~$x\in V(C)$ and~$y\in V\setminus V(C)$, the graph~$C[x\to y]$ has the vertex set~$V(C)\setminus\{x\}\cup\{y\}$ and edge set~$\{uv\in E(C):\: u,v\neq x\} \cup \{yw:\:xw\in E(C)\}$.
	We also say~$C[x\to y]$ is the graph obtained from~$C$ by \emph{replacing}~$x$ by~$y$.

    Let~$D$ be a directed graph and~$v \in V(D)$ a vertex.
    The~\emph{out-neighbourhood} and~\emph{in-neighbourhood} of~$v$ are
	\[
	N^+_D(v) := \{ u \in V(D) :\: (v,u) \in E(D) \} \text{ and }
	N^-_D(v) := \{ u \in V(D) :\: (u,v) \in E(D) \}\,,
	\]
	respectively.
	Accordingly, the~\emph{out-degree} and the \emph{in-degree} of~$v$ are defined as
	\[
	\deg^+_D(v) := |N^+_D(v)|\text{ and }
	\deg^-_D(v) := |N^-_D(v)|\,,
	\]
	respectively.
	
	A~\emph{tournament} on~$n$ vertices is a directed graph~$T$ in which for every pair of distinct vertices~$u,v \in V(T)$ exactly one of the two edges~$(u,v)$ or $(v,u)$ is present.
	That is, a tournament is an orientation of a complete (undirected) graph where each edge is assigned a direction.
	
	\section{Proof of Theorem~\ref{thm:main result}}\label{sec:main proof}
	Let~$k\geq3$ be an integer, let~$G$ be a~$2$-graph on~$t$ vertices, and let~$H$ be an arbitrary~$k$-graph on~$n\geq n_0$ vertices, where~$n_0$ is sufficiently large.
	Further, let~$(H_i)_{i \geq 0}$ be the~$F^{(k)}(G)$-process starting with~$H$.
	Set~$V=V(H)$ and for later use we choose integers~$\ell$,~$c$,~$s$,~$\Gamma_2$,~$\Gamma_1$, and $\Gamma$ such that
	\begin{align}\label{eq:hierarchy}
		t,k \ll \ell \ll c \ll s \ll \Gamma_2 \ll \Gamma_1 \ll \Gamma \ll n_0\,.
	\end{align}
	
	Moreover, let~$\iota=\tau_F(H)$ be the running time of~$(H_i)_{i\geq 0}$.
    We show that assuming~$\iota>2\Gamma$ eventually leads to a contradiction.
    Note that it is enough to establish this bound for~$n\geq n_0$ since we can take the maximum of~$2\Gamma$ and~$\max_{n'< n_0}M_{F^{(k)}(G)}(n')$ to obtain a constant upper bound on~$M_{F^{(k)}(G)}(n)$ for all~$n$.
    
	The following simple observation will be useful throughout the proof.
	
	\begin{observation}\label{obs:almostfulldegree}
		Let~$F$ be a copy of~$F^{(k)}(G)$ in~$H_m$, for some~$m \geq 0$, and let~$uv \in E(C(F))$.
        Further, let~$\mathbf{w} \subseteq V \setminus V(F)$ and~$\mathbf{w}' \subseteq V \setminus V(C(F))$ be~$(k-2)$-sets.
        Then we have~$u\mathbf{w}v\in E(H_{m+1})$ and~$u\mathbf{w}'v\in E(H_{m+2})$.
		%Consequently, for every~$(k-3)$-set~$\mathbf{x} \subseteq V \setminus V(C(F))$, we have~$d_{H_{m+2}}(u\mathbf{x}v)\geq n-t-k+3$.
	\end{observation}
	
	\begin{proof}
		
		First, observe that for any edge~$xy \in E(C(F))$ and any~$(k-2)$-set~$\mathbf{w} \subseteq V \setminus V(F)$, it follows that~$x\mathbf{w}y\in E(H_{m+1})$ by the definition of the~$F^{(k)}(G)$-process.
		Hence, since~$n$ is sufficiently large, in~$H_{m+1}$ there is a copy of~$F^{(k)}(G)$ with (up to isomorphism) the same center as~$F$ whose vertex set is disjoint from~$\mathbf{w'}$.
		Hence,~$u\mathbf{w'}v \in E(H_{m+2})$ by the definition of the~$F^{(k)}(G)$-process.
		%The third part follows immediately from the first two.
	\end{proof}

    Let~$\cC(H_m)$ denote the family of centers of copies of~$F^{(k)}(G)$ in~$H_m$ and let~$E(\cC(H_m))\subseteq V^{(2)}$ be the collection of all pairs of vertices that form an edge in some~$C\in\cC(H_m)$.
	Given~$m\in\mathds{N}_0$, a vertex~$v \in V$ is called \emph{$m$-good} 
	if there are at least~$c$ distinct vertices~$w\in V$ with~$vw\in E(\cC(H_m))$.
	The set of~$m$-good vertices is denoted by~$V_m^{\text{good}}$.
	We say that~$m\in\mathds{N}$ is a \emph{real step} if~$E(\cC(H_m))\setminus E(\cC(H_{m-1}))\neq\emptyset$.
	Otherwise, we say that~$m$ is a \emph{fake step}.
	
	Next, we define the notion of~$m$-essential pairs for~$m\in\mathds{N}$ being a real step.
	For each real step~$m$, arbitrarily choose some copy~$F_m$ of~$F^{(k)}(G)$ in~$H_m$ such that~$E(C(F_m))\setminus E(\cC(H_{m-1}))$ is non-empty.
	For any pair of vertices~$xy\in V^{(2)}$ let~$d^{\text{match}}_{H_m}(xy)$ denote the maximum number of pairwise disjoint sets~$\mathbf{u}\subseteq V^{(k-2)}$ with~$x\mathbf{u}y\in E(H_m)$.
	In other words,~$d^{\text{match}}_{H_m}(xy)$ is the maximum size of a matching in the family~$\{\mathbf{u}\subseteq V^{(k-2)}:\:x\mathbf{u}y\in E(H_m)\}$.
	Now amongst all~$e\in E(C(F_m))$ select an element~$e$ which minimises~$d^{\text{match}}_{H_{m-1}}(e)$ and denote it by~$e_m$.
	We call~$e_m$ the~$m$-essential pair and~$F_m$ the~$m$-essential copy associated with~$e_m$.
	A pair is said to be \emph{essential} if it is~$m$-essential for
	some integer~$m \ge 1$.
	Note that~$F_m$ is contained in~$H_m$ but not in~$H_{m-1}$.
	
	The following observation guarantees many real steps.
	
	\begin{observation}\label{obs:1/3realstep}
		If~$m$ is an integer with~$3\leq m\leq\iota$, then one of~$m-2$,~$m-1$, and~$m$ must be a real step.
	\end{observation}
	
	\begin{proof}
		For the sake of contradiction, assume the opposite, that is,
		\begin{align}\label{eq:nonewcenteredge}
			E(\cC(H_m))=E(\cC(H_{m-1}))=E(\cC(H_{m-2}))=E(\cC(H_{m-3}))\,.
		\end{align}
		Since~$m\leq\iota$, we further know that there is a copy~$F$ of~$F^{(k)}(G)$ which is in~$H_m$ but not in~$H_{m-1}$.
		Let~$C$ be the center of~$F$.
		Due to~\eqref{eq:nonewcenteredge}, for each~$e\in E(C)$, there is some copy~$F'$ of~$F^{(k)}(G)$ in~$H_{m-3}$ such that~$e\in C(F')$.
        Hence, by Observation~\ref{obs:almostfulldegree}, we have~$e\mathbf{u}\in E(H_{m-2})$ for every~$(k-2)$-set~$\mathbf{u}\subseteq V\setminus V(F')$.
        Considering~\eqref{eq:hierarchy}, this means that there is a copy~$\widetilde{F}$ of~$F^{(k)}(G)$ in~$H_{m-2}$ with~$C(\widetilde{F})=C$ such that~$V(\widetilde{F})$ is disjoint from~$V(F)\setminus V(C)$.
		Therefore, Observation~\ref{obs:almostfulldegree} guarantees that~$F$ is already present in~$H_{m-1}$, a contradiction.
	\end{proof}
	
	Next, we observe that the same pair can be chosen at most twice as essential pair.
    \begin{observation}\label{obs:atmosttwiceessential}
		For every~$e\in V^{(2)}$ there are at most two distinct indices~$m_1,m_2\in[\iota]$ with~$e=e_{m_1}=e_{m_2}$.    
	\end{observation}
	\begin{proof}
		Suppose that there are three distinct indices~$m_1,m_2,m_3\in[\iota]$ with~$m_1<m_2<m_3$ such that~$e=e_{m_1}=e_{m_2}=e_{m_3}$ for some~$e\in V^{(2)}$.
		By the definition of~$m_1$-essential,~$e_{m_1}$ is an edge of some center in~$\cC(H_{m_1})$.
		Thus, the definition of the~$F^{(k)}(G)$-process entails~$d^{\text{match}}_{H_{m_1+1}}(e_{m_1})\geq n/(2(k-2))$.
		Since~$e_{m_3}$ is~$m_3$-essential, we have $$E(C(F_{m_3}))\setminus E(\cC(H_{m_3-1}))\neq \emptyset\,,$$ and, for every~$p\in E(C(F_{m_3}))$, $$d^{\text{match}}_{H_{m_3-1}}(p)\geq d^{\text{match}}_{H_{m_3-1}}(e_{m_3}) \geq d^{\text{match}}_{H_{m_1+1}}(e_{m_1})\geq n/(2(k-2))\,.$$
		By~\eqref{eq:hierarchy}, we can hence pick pairwise disjoint~$(k-2)$-sets~$\mathbf{y}_p\subseteq V\setminus V(C(F_{m_3}))$, for each~$p\in E(C(F_{m_3}))$, such that~$p\mathbf{y}_p\in E(H_{m_3-1})$.
		In other words,~$C(F_{m_3})\in\cC(H_{m_3-1})$, whence~$E(C(F_{m_3}))\subseteq E(\cC(H_{m_3-1}))$.
		This is a contradiction.
	\end{proof}

    The following observation is needed to make use of good vertices.

    \begin{observation}\label{obs:manycenteredges->good}
        If~$v\in V$ is~$m$-good for some~$m\in\mathds{N}$, then there is a family~$\cT_v\subseteq V^{(k-2)}$ of pairwise disjoint~$(k-2)$-sets with~$\vert\cT_v\vert=c$ such that~$d_{H_{m+1}}(v\mathbf{u})\geq n-c$ for every~$\mathbf{u}\in\cT_v$.
    \end{observation}
    
    \begin{proof}
        Since~$v$ is~$m$-good, there are~$c$ distinct vertices~$w_1, \ldots, w_c$ with~$vw_i\in E(\cC(H_m))$.
        So for each~$i\in[c]$, there is some copy~$\widetilde{F}_i$ of~$F^{(k)}(G)$ in~$H_m$ with~$vw_i\in C(\widetilde{F}_i)$.
		Therefore, Observation~\ref{obs:almostfulldegree} entails that for every~$(k-3)$-set~$\mathbf{x}_i \subseteq V \setminus V(\widetilde{F}_i)$, the set~$vw_i\mathbf{x}_i$ has degree at least~$n-t-(k-2)\binom{t}{2}-(k-3)>n-c$ in~$H_{m+1}$.
		By~\eqref{eq:hierarchy}, we can choose~$\mathbf{x}_1,\ldots,\mathbf{x}_t$ so that they are pairwise disjoint and so that each of them is disjoint from~$\{w_1,\dots,w_c\}$.
		Thus, the~$c$ pairwise disjoint~$(k-2)$-subsets~$w_i\mathbf{x}_i$ are as desired.
    \end{proof}

    For each~$m$-good vertex~$v$, fix a family~$\cT_v$ as in Observation~\ref{obs:manycenteredges->good}.
    The next lemma states that the process will stop soon after there are many good vertices.
    
	\begin{lemma} \label{lem:good+6}
		If~$\vert V_m^{\text{good}}\vert\geq s$ for some~$m\in\mathds{N}_0$, then~$\iota\leq m+6$.
	\end{lemma}
	
	\begin{proof}
		
		We choose pairwise disjoint~$\ell$-subsets~$W_1, \ldots, W_t \subseteq V_m^{\text{good}}$ such that for all integers~$1\leq i<j\leq t$ and for all~$v\in W_i$,~$x\in W_j$, and~$\mathbf{u}\in\cT_v$ we have $v\mathbf{u} x\in E(H_{m+1})$.
		To see that this is possible, first select~$W_1\subseteq V_m^{\text{good}}$ with~$\vert W_1\vert=\ell$ arbitrarily.
		Then assume that~$W_1,\dots, W_{i}$ have been chosen for some~$i\in[t-1]$.
		Note that a vertex~$w\in V_m^\text{good}$ can only not be selected for~$W_{i+1}$ if it lies in~$W_1\cup\dots\cup W_i$ or if there is some vertex~$v\in W_1\cup\dots\cup W_i$ and some~$\mathbf{u}\in \cT_v$ such that~$v\mathbf{u}w\notin E(H_{m+1})$.
        Hence, in light of Observation~\ref{obs:manycenteredges->good}, there are at most~$i\cdot \ell+i\cdot \ell\cdot c^2$ vertices in~$V_m^\text{good}$ which cannot be selected as vertices for~$W_{i+1}$.
		Due to~\eqref{eq:hierarchy}, we can thus still choose~$W_{i+1}\subseteq V_m^{\text{good}}$ as desired.
		
		Arbitrarily pick~$w_i^\star \in W_i$ for each~$i \in [t]$.
		We claim that there are pairwise disjoint~$(k-2)$-sets~$\mathbf{u}_{ij}\in\cT_{w_i^\star}$, for all~$1\leq i<j\leq t$, such that~$\mathbf{u}_{ij}\cap\bigcup_{a\in[t]}W_a=\emptyset$.
		Choose the sets~$\mathbf{u}_{ij}$ greedily following the lexicographic order of~$[t]^{(2)}$, which we denote by~$\preceq_{\text{lex}}$. 
		When picking~$\mathbf{u}_{ij}\in\cT_{w_i^\star}$ at any point during this selection, we only need to guarantee that it is disjoint from the set $$B=\bigcup_{a\in[t]}W_a\cup\bigcup_{i'j'\in[t]^{(2)}\text{ with }i'j'\prec_{\text{lex}} ij}\mathbf{u}_{i'j'}\,.$$
		Since~$\vert B\vert\leq t\cdot\ell+\binom{t}{2}\cdot(k-2)\ll c$ and the sets in~$\cT_{w_i^\star}$ are pairwise disjoint, there is indeed a valid choice for~$\mathbf{u}_{ij}$.
		
		Given~$i,j\in[t]$ with~$i<j$, the choice of~$W_i$ and~$W_j$ ensures that for every~$\mathbf{u} \in \cT_{w_i^\star}$ we have~$w_i^\star \mathbf{u} w_j^\star \in E(H_{m+1})$, whence~$w_i^\star \mathbf{u}_{ij} w_j^\star \in E(H_{m+1})$.
		Consequently,~$H_{m+1}$ contains a copy~$F$ of~$F^{(k)}(K_t^{(2)})$ with~$V(C(F))=\{w_1^\star, \ldots, w_t^\star\}$.
		In fact, since~$w_1^\star,\dots,w_t^\star$ were chosen arbitrarily, for every selection of vertices~$w_i\in W_i$,~$i\in[t]$,~$H_{m+1}$ contains a copy of~$F^{(k)}(K_t^{(2)})$ such that the vertex set of the center of this copy is~$\{w_1, \ldots, w_t\}$.
		By Observation~\ref{obs:almostfulldegree},
		\begin{align}\label{eq:edgescliqueextension}
			\text{$w_i\mathbf{u}w_j$ forms an edge in~$H_{m+3}$}
		\end{align}
		for all~$ij\in[t]^{(2)}$,~$w_i\in W_i$,~$w_j\in W_j$, and every~$(k-2)$-set~$\mathbf{u} \subseteq V \setminus \{w_i,w_j\}$.
		
        Now let~$a,b\in V$ be arbitrary and pick~$w_i\in W_i\setminus\{a,b\}$, for all~$i\in[t-2]$.
		Note that~\eqref{eq:edgescliqueextension} entails that for every~$w_{t-1}\in W_{t-1}\setminus\{a\}$, every~$i\in[t-2]$, and every~$(k-3)$-set~$\mathbf{u'}\subseteq V\setminus\{a,w_i,w_{t-1}\}$ we have~$w_i\mathbf{u'}aw_{t-1}\in E(H_{m+3})$.
		Similarly, for every~$w_t\in W_t\setminus\{b\}$, every~$i\in[t-2]$ and every~$(k-3)$-set~$\mathbf{u'}\subseteq V\setminus\{b,w_i,w_t\}$ we have~$w_i\mathbf{u'}bw_t\in E(H_{m+3})$.
		Using the hierarchy~\eqref{eq:hierarchy}, one can thus observe that it is possible to choose~$\widetilde{\mathbf{u}}_{ij}\subseteq V^{(k-2)}$,~$\mathbf{u}'_i,\mathbf{u}''_i\subseteq V^{(k-3)}$, and~$w_{t-1}^i\in W_{t-1}$ as well as~$w_t^i\in W_t$, for all~$i,j\in[t-2]$, such that the following holds.
        The sets~$\widetilde{\mathbf{u}}_{ij},\mathbf{u}'_i,\mathbf{u}''_i,\{w_{t-1}^i\},\{w_t^i\}$, and~$\{w_1,\dots,w_{t-2},a,b\}$ are all pairwise disjoint and~$w_i\widetilde{\mathbf{u}}_{ij}w_j\in E(H_{m+3})$,~$w_i\mathbf{u}'_iw_{t-1}^ia\in E(H_{m+3})$, and~$w_i\mathbf{u}''_iw_t^ib\in E(H_{m+3})$.
        This means that in~$H_{m+3}$ there is a copy of~$F^{(k)}(K_t^{(2)-})$ whose center has the vertex set~$\{w_1,\dots,w_{t-2},a,b\}$ with~$ab$ being the missing edge.
		Hence,~$H_{m+4}$ must contain a copy of~$F^{(k)}(K_t^{(2)})$ whose center has the same vertex set.
		As argued above (see~\eqref{eq:edgescliqueextension}), this implies that for every~$(k-2)$-subset~$\mathbf{u}\subseteq V\setminus\{a,b\}$ we have~$a\mathbf{u}b\in E(H_{m+6})$.
		Since~$a$ and~$b$ were arbitrary vertices, we infer that~$H_{m+6}$ is complete, whereby~$\iota\leq m+6$.
	\end{proof}

	Given a vertex~$v$ and~$m\in\mathds{N}$, 
    let~$\cE_{\leq m}$ be the set of essential pairs that are~$i$-essential for some~$i\in[m]$, and let~$\cE_{\leq m}^v$ be the subset of those pairs incident to~$v$.
    The definition of an~$m$-good vertex gives that any vertex~$v$ with~$\vert\cE_{\leq m}^v\vert\geq c$ is~$m$-good.
    Therefore, Lemma~\ref{lem:good+6} states that we are done if there are at least~$s$ distinct vertices~$v$ with~$\vert\cE_{\leq\Gamma}^v\vert\geq c$.
    To show that there are many good vertices, it is helpful to argue that no vertex can be contained in too many essential pairs.
    
	\begin{lemma} \label{lem:essential+bounded}
		For each vertex~$v\in V$, there are at most~$c^3$ essential pairs containing~$v$.
	\end{lemma}
	\begin{proof}
		If there are fewer than~$c$ essential pairs containing~$v$, there is nothing to show.
		Let~$m$ be the smallest integer such that~$\vert\cE_{\leq m}^v\vert= c$.
		Then~$v$ is~$m$-good by definition.
		
		Now we will argue that for every~$x\in \bigcap_{\mathbf{y}\in\cT_v}N_{H_{m+1}}(v\mathbf{y})$ the pair~$vx$ cannot be~$m'$-essential for any~$m'>m+1$.
		For the sake of contradiction assume the contrary.
		Then the~$m'$-essential copy~$F_{m'}$ of~$F^{(k)}(G)$ in~$H_{m'}$ satisfies~$vx\in E(C(F_{m'}))$ and~$E(C(F_{m'}))\setminus E(\cC(H_{m'-1}))\neq\emptyset$.
		Moreover, by the definition of~$m'$-essential, it must hold for every~$p\in E(C(F_{m'}))$ that $$d^{\text{match}}_{H_{m'-1}}(p)\geq d^{\text{match}}_{H_{m'-1}}(vx)\geq d^{\text{match}}_{H_{m+1}}(vx)\geq\vert\cT_v\vert=c\,.$$
		By~\eqref{eq:hierarchy}, we can hence pick pairwise disjoint~$(k-2)$-sets~$\mathbf{y}_p\subseteq V\setminus V(C(F_{m'}))$, for each~$p\in E(C(F_{m'}))$, such that~$p\mathbf{y}_p\in E(H_{m'-1})$.
		In other words,~$C(F_{m'})\in\cC(H_{m'-1})$, whence~$E(C(F_{m'}))\subseteq E(\cC(H_{m'-1}))$.
		This is a contradiction.
		
		Therefore, (for~$m'>m+1$)~$v$ can only form an~$m'$-essential pair with vertices in $V\setminus\bigcap_{\mathbf{y} \in \cT_v} N_{H_{m+1}}(v\mathbf{y})$, and the size of this set is at most~$c^2$. 
		Thus, the number of all essential pairs containing~$v$ is at most~$c +1+ c^2 < c^3$.
	\end{proof} 

    Next we show that there is a large matching of essential pairs with some nice properties.
    
	\begin{lemma}\label{lem:matchingessential}
		If~$\iota>\Gamma$, then there is a subset~$J\subseteq[\Gamma]$ with~$\vert J\vert=\Gamma_1$ such that
		\begin{enumerate}[label=$(\arabic*)$]
			\item \label{it:essentialdisjointothercopies} $e_i \cap V(F_j) = \emptyset$ for all distinct~$i,j\in J$, and
			\item \label{it:playthesamerole} $e_i$ plays the same r\^ole in~$F_i$ as~$e_j$ does in~$F_j$ for all~$i,j\in J$.
		\end{enumerate}
	\end{lemma}
	
	\begin{proof}
        If~$\iota>\Gamma$, then it follows from Observation~\ref{obs:1/3realstep} and Observation~\ref{obs:atmosttwiceessential} that~$\vert\cE_{\leq\Gamma}\vert\geq\Gamma/7$.
        Thus we can choose~$\widetilde{J}\subseteq[\Gamma]$ with~$\vert\widetilde{J}\vert=\Gamma/7$ such that~$e_i\neq e_j$ for all~$i,j\in\widetilde{J}$ with~$i\neq j$.
		Now consider an auxiliary graph~$A$ on the vertex set~$\widetilde{J}$ where~$ij$ forms an edge if~$e_j\cap V(F_i)=\emptyset=e_i\cap V(F_j)$.
        By Lemma~\ref{lem:essential+bounded}, for every~$v\in V$ there are at most~$c^3$ indices~$i\in \widetilde{J}$ with~$v\in e_i$.
        Consequently, for every~$i\in\widetilde{J}$ there are at most~$\big(t+(k-2)\binom{t}{2}\big)\cdot c^3\leq c^4$ indices~$j\in\widetilde{J}$ with~$e_j\cap V(F_i)\neq \emptyset$.
        This in turn implies that~$\vert E(A)\vert\geq\binom{\vert\widetilde{J}\vert}{2}-\vert\widetilde{J}\vert\cdot c^4$.
        Due to Tur\'an's theorem and the hierarchy~\eqref{eq:hierarchy}, there is a clique of size~$\Gamma_1^2$ in~$A$.
        Note that this is the same as saying that there is some~$\hat{J}\subseteq\widetilde{J}$ with~$\vert\hat{J}\vert=\Gamma_1^2$ satisfying property~\ref{it:essentialdisjointothercopies}.
        
		Since there are at most~$\binom{t}{2}$ edges in~$G$, an averaging argument guarantees the existence of a subset~$J\subseteq\hat{J}$ such that~$\vert J\vert=\Gamma_1$ and~$e_i$ is in the same position of~$F_i$ for all~$i\in J$.
	\end{proof}

    Since we are done if~$\iota\leq\Gamma$, we may assume that a set~$J$ as in Lemma~\ref{lem:matchingessential} indeed exists.
	For each~$i\in J$, let~$e_i=a_ib_i$ such that the position of~$a_i$ in $F_i$ is the same that of~$a_j$ in~$F_j$ for all~$i\neq j$.
	
	\begin{observation}\label{obs:manyalternatives+good}
		If for some~$i\in J$ and~$m\in\mathds{N}$ there are at least~$c$ distinct indices~$j\in J\setminus\{i\}$ such that~$C(F_i)[a_i \to a_j]\in\cC(H_m)$, then~$b_i$ is~$m$-good.
	\end{observation}
	
	\begin{proof}
		Let~$J_i$ be a set of~$c$ such indices~$j$.
		For each~$j\in J_i$, since~$C(F_i)[a_i\to a_j]\in\cC(H_m)$, we have~$a_jb_i\in E(\cC(H_m))$.
        Thus,~$b_i$ is~$m$-good by definition.
	\end{proof}
	
    The previous observation allows us to finish the proof if for many~$ij\in J^{(2)}$ replacing~$a_i$ by~$a_j$ in~$C(F_i)$ yields another center.
    The next lemma reveals that we indeed also gain something if this does not hold.
    
	\begin{lemma}\label{lem:cannotreplace}
		If~$C(F_i)[a_i \to a_j]\notin\cC(H_{\Gamma+2})$ for some~$i,j\in J$, then
		\[
		N_{C(F_i)}(a_i)\cap\bigl(V(C(F_j))\setminus N_{C(F_j)}(a_j)\bigr)\neq\emptyset
		\text{ or }
		N_{C(F_j)}(a_j)\cap\bigl(V(C(F_i))\setminus N_{C(F_i)}(a_i)\bigr)\neq\emptyset\,.
		\]
	\end{lemma}
	
	\begin{proof}
		Suppose to the contrary that $$N_{C(F_i)}(a_i)\cap\bigl(V(C(F_j))\setminus N_{C(F_j)}(a_j)\bigr)=\emptyset=N_{C(F_j)}(a_j)\cap\bigl(V(C(F_i))\setminus N_{C(F_i)}(a_i)\bigr)\,.$$
		Let~$u_1:=b_j$ and~$v_1:=b_i$ and recall that~$b_j\notin V(F_i)$ and~$b_i\notin V(F_j)$.
		Hence, for some~$r\in[t-1]$ and some~$p\in[r]$, we can write
		\[
		N_{C(F_j)}(a_j)=\{u_1,\ldots,u_r\}\,,\qquad
		N_{C(F_i)}(a_i)=\{v_1,\ldots,v_r\}\,,
		\]
		such that~$u_h=v_h$ for~$h\in[p+1,r]$ and $$\{u_{p+1},\dots,u_{r}\}=\{v_{p+1},\dots,v_{r}\}=N_{C(F_i)}(a_i)\cap N_{C(F_j)}(a_j)\,.$$

		 Since~$C(F_j)\in\cC(H_{\Gamma})$ and because by assumption~$\{v_1,\dots,v_p\}\cap V(C(F_j))=\emptyset$, we can utilise Observation~\ref{obs:almostfulldegree} to find pairwise disjoint~$(k-3)$-sets
		\[
		\mathbf w_h
		\subseteq V\setminus
		\Bigl(V(F_i)\cup V(F_j)\Bigr)\,,
		\]
		for each~$h\in[p]$, such that~$a_ju_h\mathbf w_hv_h\in E(H_{\Gamma+2})$.
        Moreover, because~$u_h=v_h$ for~$h\in[p+1,r]$ and~$C(F_j)\in\cC(H_{\Gamma})$, Observation~\ref{obs:almostfulldegree} allows us to find pairwise disjoint~$(k-2)$-sets $$\mathbf{x}_h\subseteq V\setminus\Big(V(F_i)\cup V(F_j)\cup\bigcup_{h\in[p]}\mathbf{w}_h\Big)\,,$$ for each~$h\in[p+1,r]$ such that~$a_jv_h\mathbf{x}_h\in E(H_{\Gamma+1})$.
    	Due to~$\{u_1,\dots,u_p\}\cap V(C(F_i))=\emptyset$, this together indeed yields~$C(F_i)[a_i\to a_j]\in\cC(H_{\Gamma+2})$, a contradiction.
	\end{proof}

    We construct an auxiliary directed graph~$D$ with vertex set~$V(D)=J$.
	For any two distinct vertices~$i,j\in J$,~$(j,i)$ forms an edge in~$D$ if~$N_{C(F_i)}(a_i)\cap \big(V(C(F_j))\setminus N_{C(F_j)}(a_j)\big)$ is empty and there is some vertex
    \begin{align}\label{eq:meaningof(i,j)}
        x\in N_{C(F_j)}(a_j)\cap \big(V(C(F_i))\setminus N_{C(F_i)}(a_i)\big)\,.
    \end{align}
    If both~$N_{C(F_i)}(a_i)\cap \big(V(C(F_j))\setminus N_{C(F_j)}(a_j)\big)$ and~$N_{C(F_j)}(a_j)\cap \big(V(C(F_i))\setminus N_{C(F_i)}(a_i)\big)$ are non-empty, we arbitrarily choose one of~$(i,j)$ and~$(j,i)$ to be an edge.
    Further, let~$D^{\star}$ be the graph on~$V(D)$ in which~$ij$ forms an edge if~$(i,j)$ or~$(j,i)$ is an edge in~$D$.
    Note that Lemma~\ref{lem:cannotreplace} implies that if~$ij\notin E(D^{\star})$, then~$C(F_i)[a_i \to a_j]\in\cC(H_{\Gamma+2})$.
    Therefore, if in~$D^{\star}$ there are at least~$c$ edges missing at some~$i\in V(D)$, then by Observation~\ref{obs:manyalternatives+good},~$b_i$ is~$(\Gamma+2)$-good.
    Hence, by Lemma~\ref{lem:good+6}, if there are at least~$s$ such~$i\in V(D)$, it follows that~$\iota\leq\Gamma+8$ and we are done.
    In the following, we complete the proof by deriving a contradiction if the opposite holds.
    
    So now assume that the number of~$i\in V(D)$ with~$d_{D^{\star}}(i)\leq \Gamma_1-c$ is less than~$s$.
    Notice that in this case, we have $$\vert E(D^{\star})\vert\geq\binom{\Gamma_1}{2}-\big((s-1)\Gamma_1 + \Gamma_1(c-1)\big) \geq \binom{\Gamma_1}{2}-(s+c)\Gamma_1\,.$$
    
	By Turán’s theorem, this implies the existence of a~$K_{\Gamma_2}$ in~$D^{\star}$, that is, a tournament~$D_1\subseteq D$ on~$\Gamma_2$ vertices.
    Choose some~$i_1\in V(D_1)$ with~$d^{-}_{D_1}(i_1)\geq \frac{\Gamma_2}{3}$, which is possible by averaging.
    By~\eqref{eq:meaningof(i,j)}, for each~$j\in N^{-}_{D_1}(i_1)$, there is some~$x\in N_{C(F_j)}(a_j)\cap \big(V(C(F_{i_1}))\setminus N_{C(F_{i_1})}(a_{i_1})\big)$.
    Hence, since there are at most~$t$ possibilities for~$x$, there are~$x_1\in V(F_{i_1})$ and a set~$I_1\subseteq N^{-}_{D_1}(i_1)$ with~$\vert I_1\vert\geq\frac{\Gamma_2}{3t}$ such that~$x_1\in N_{C(F_j)}(a_j)\cap \big(V(C(F_{i_1}))\setminus N_{C(F_{i_1})}(a_{i_1})\big)$ for all~$j\in I_1$.
    Now assume that for some~$\alpha\in[s-1]$ we have found~$i_1,\dots,i_{\alpha}\in V(D_1)$,~$x_1,\dots,x_{\alpha}\in V(H)$ and sets~$I_1,\dots,I_{\alpha}\subseteq V(D_1)$ with the following properties
    \begin{enumerate}
        \item\label{it:nested} $i_{\beta+1}\in I_{\beta}$ and~$I_{\beta+1}\subseteq I_{\beta}$ for all~$\beta\in[\alpha-1]$,
        \item\label{it:size indeg} $\vert I_{\beta}\vert\geq\frac{\Gamma_2}{(3t)^{\beta}}$ for all~$\beta\in[\alpha]$, and
        \item\label{it:vertexcontainment} $x_{\beta}\in N_{C(F_j)}(a_j)\cap\big(V(C(F_{i_{\beta}}))\setminus N_{C(F_{i_{\beta}})}(a_{i_{\beta}})\big)$ for all~$\beta\in[\alpha]$ and~$j\in I_{\beta}$.
    \end{enumerate}

    Since~$D_1$ is a tournament, averaging yields some~$i_{\alpha+1}\in I_{\alpha}$ with~$\vert N_{D_1}^-(i_{\alpha+1})\cap I_{\alpha}\vert\geq\frac{\vert I_{\alpha}\vert}{3}\geq\frac{\Gamma_2}{3(3t)^{\alpha}}$.
    As above, by~\eqref{eq:meaningof(i,j)} and averaging there is some~$x_{\alpha+1}$ and a set~$I_{\alpha+1}\subseteq N_{D_1}^-(i_{\alpha+1})\cap I_{\alpha}$ with~$\vert I_{\alpha+1}\vert\geq\frac{\Gamma_2}{(3t)^{\alpha+1}}$ such that for all~$j\in I_{\alpha+1}$ $$x_{\alpha+1}\in N_{C(F_j)}(a_j)\cap\big(V(C(F_{i_{\alpha+1}}))\setminus N_{C(F_{i_{\alpha+1}})}(a_{i_{\alpha+1}})\big)\,.$$
    Considering~\eqref{eq:hierarchy}, we thus find~$i_1,\dots,i_s$,~$x_1,\dots,x_s$, and~$I_1,\dots,I_s$ with the above properties.

    Note that properties~\eqref{it:nested} and~\eqref{it:vertexcontainment} imply that all the~$x_{\beta}$ are distinct.
    This is a contradiction since by~\eqref{it:vertexcontainment} (and~\eqref{it:nested}) $a_{i_s}x_{\beta}\in E(C(F_{i_s}))$, for all~$\beta\in[s-1]$, but~$\vert V(C(F_{i_s}))\vert=t\ll s$.
    \qed
	
	\section{Proof of Proposition~\ref{prop:not absolute constant lower bound}}\label{sec:prop}
	
	Theorem~\ref{thm:main result} shows that for any graph~$G$ and~$n\in\mathds{N}$ the maximum running time~$M_{F^{(k)}(G)}(n)$ is bounded by a constant that depends only on~$G$ and~$k$.
    For some special classes of graphs~$G$ (for instance paths and cycles), there are absolute constants bounding the asymptotics of~$M_{F^{(k)}(G)}(n)$.
    However, Proposition~\ref{prop:not absolute constant lower bound} states that there are some graphs~$G$ for which~$M_{F^{(k)}(G)}(n)$ is at least linear in~$\vert V(G)\vert$.
    The construction that yields the bound in Proposition~\ref{prop:not absolute constant lower bound} is a direct generalisation of a construction in~\cite{FMS:25}, which provides a lower bound on~$M_T(n)$ for~$T$ being a ($2$-uniform) star.
	
	\noindent{\bf Proof of Proposition~\ref{prop:not absolute constant lower bound}.}
    Given integers~$k\geq3$ and~$t\geq 3$, let~$T$ be~$K_{1,t-1}$, i.e., the star on~$t$ vertices,~$F^{(k)}(T)$ its~$k$-extension, and denote the unique vertex of degree~$t-1$ in~$F^{(k)}(T)$ by~$c$.
    We need to show that for every large enough~$n\in\mathds{N}$ there is some~$H$ such that the~$F^{(k)}(T)$-process starting with~$H$ has running time at least~$t-1$.
    
    For every~$h\in[t-2]$ let~$F_h$ be a copy of~$F^{(k)}(K_{1,h})$ (on a distinct vertex set) and denote the vertex of degree~$h$ by~$v_h$.
	Given some large~$n\in\mathds{N}$, let~$H$ be the~$k$-graph on~$n$ vertices which is the disjoint union of the~$F_h$, for all~$h\in[t-2]$, and a set~$W$ of~$n-\sum_{h=1}^{t-2}\big(h+1+h(k-2)\big)$ isolated vertices.
	Setting~$H_0=H$, we claim that the~$F^{(k)}(T)$-process~$(H_i)_{i \ge 0}$ runs for at least~$t-1$ steps.
    
    Given~$u\in V(H)$ and~$i\in\mathds{N}_{\geq0}$, let~$h_i(u)$ denote the largest~$h$ such that~$u$ plays the r\^ole of the vertex of degree~$h$ in some copy of~$F^{(k)}(K_{1,h})$ in~$H_i$.
    By induction on~$i$ we show that for every~$i\in[0,t-3]$ we have the following properties
    \begin{enumerate}
        \item\label{it:nottoomany} every~$e\in E(H_i)\setminus E(H_0)$ contains one of~$v_{t-i-1},\dots,v_{t-2}$,
        \item\label{it:nottoofew} $v_jv_{\ell}\mathbf{w}\in E(H_i)$, for all~$j\in[t-1-i,t-2]$,~$\ell\in[t-2]\setminus\{j\}$, and~$\mathbf{w}\in W^{(k-2)}$, and
        \item\label{it:nottoofew2} $v_j\mathbf{w}\in E(H_i)$, for all~$j\in[t-1-i,t-2]$ and~$\mathbf{w}\in W^{(k-1)}$\,.
    \end{enumerate}
    Clearly, these properties hold for~$i=0$.
    Given~$i\in[t-3]$ assume that the above properties hold for $i-1$.
    Note that~\eqref{it:nottoofew} for~$i-1$ and the definition of~$H$ entail that for~$\mathbf{w}\in W^{(k-2)}$ there is a copy~$F$ of~$F^{(k)}(K_{1,t-2})$ in~$H_{i-1}$, with~$v_{t-1-i}$ playing the r\^ole of the vertex of degree~$t-2$, such that~$V(F)\cap\mathbf{w}=\emptyset$.
    Indeed, using that~\eqref{it:nottoofew} holds for~$i-1$,~$v_{t-i},\dots,v_{t-2}$ can be added as new leaves to the star that forms the center of~$F_{t-1-i}$ in such a way that the resulting copy is vertex-disjoint to~$\mathbf{w}$ because~$n$ is large enough.
    Consequently, the addition of any~$k$-set of the form~$v_{t-1-i}v_{\ell}\mathbf{w}$ with~$\ell\in[t-2-i]$ and~$\mathbf{w}\in W^{(k-2)}$ as an edge to~$H_{i-1}$ creates a new copy of~$F^{(k)}(T)$.
    Together with~\eqref{it:nottoofew} for~$i-1$, this implies that~\eqref{it:nottoofew} holds for~$i$.
    Since also the addition of any~$k$-set of the form~$v_{t-1-i}\mathbf{w}$ for~$\mathbf{w}\in W^{(k-1)}$ as an edge to~$H_{i-1}$ creates a new copy of~$F^{(k)}(T)$, we similarly obtain~\eqref{it:nottoofew2}.

    By~\eqref{it:nottoomany} for~$i-1$, every edge in~$E(H_{i-1})\setminus E(H_0)$ contains one of~$v_{t-i},\dots,v_{t-2}$. 
    Further, in~$F^{(k)}(K_{1,h})$ every vertex but one is contained in only one edge.
    Therefore,~$h_{i-1}(v)\leq h_0(v)+i-1\leq t-3$ for all~$v\in V(H)\setminus\{v_{t-i-1},\dots,v_{t-2}\}$.
    Hence, every edge in~$E(H_i)\setminus E(H_0)$ contains one of~$v_{t-i-1},\dots,v_{t-2}$.
    This finishes the induction.

    Now note that~\eqref{it:nottoofew} and the argument for the induction step imply that there is a copy of~$F^{(k)}(K_{1,t-2})$ in~$H_{t-3}$ with~$v_1$ playing the r\^ole of the vertex of degree~$t-2$.
    Hence, as above we get~$v_1\mathbf{w}\in E(H_{t-2})$ for all~$\mathbf{w}\in W^{(k-1)}$.
    Combined with~\eqref{it:nottoofew2} and~$n$ being large this yields~$h_{t-2}(w)\geq t-2$ for all~$w\in W$, whereby~$W^{(k)}\subseteq E(H_{t-1})$.
    
    On the other hand, note that~\eqref{it:nottoomany} guarantees~$h_{t-3}(w)\leq t-3$ for every~$w\in W$.
    Thus, any edge in~$E(H_{t-2})\setminus E(H_{t-3})$ must contain a vertex in~$V(H)\setminus W$.
    Together with~\eqref{it:nottoomany}, this entails that~$W^{(k)}\cap E(H_{t-2})=\emptyset$.
    Combining this with the above, we get~$E(H_{t-1})\setminus E(H_{t-2})\neq\emptyset$ as desired.
	\qed

	\section*{Acknowledgements}
    We thank David Fabian, Gal Kronenberg, Patrick Morris, Jonathan Noel, and Tibor Szab\'o for helpful comments.
	Research was supported by the Postdoctoral Fellowship Program of CPSF under Grant Number GZC20252020, the China Postdoctoral Science Foundation 2025M783118 (Weichan Liu), and the Young Scientist Fellowship IBS-R029-Y7 (Bjarne Sch\"ulke).


\begin{bibdiv}
		\begin{biblist}

            \bib{BBM:12}{article}{
				title={Graph bootstrap percolation},
				author={Balogh, J{\'o}zsef},
				author={Bollob{\'a}s, B{\'e}la},
				author={Morris, Robert},
				journal={Random Structures \& Algorithms},
				volume={41},
				number={4},
				pages={413--440},
				year={2012},
				publisher={Wiley Online Library}
			}

            \bib{BKPS:19}{article}{
				title={The maximum length of ${K}_r $-{B}ootstrap {P}ercolation},
				author={Balogh, J{\'o}zsef},
				author={Kronenberg, Gal},
				author={Pokrovskiy, Alexey},
				author={Szab{\'o}, Tibor},
				journal={arXiv preprint arXiv:1907.04559},
				year={2019},
				pages={To appear in Proceedings of the American Mathematical Society}
			}
			
			\bib{B:68}{article}{
				author={Bollob\'as, B\'ela},
				title={Weakly $k$-saturated graphs},
				conference={
					title={Beitr\"age zur Graphentheorie},
					address={Kolloquium, Manebach},
					date={1967},
				},
				book={
					publisher={B. G. Teubner Verlagsgesellschaft, Leipzig},
				},
				date={1968},
				pages={25--31},
			}

            % \bib{BPRS:17}{article}{
            %    author={Bollob\'as, B\'ela},
            %    author={Przykucki, Micha\l},
            %    author={Riordan, Oliver},
            %    author={Sahasrabudhe, Julian},
            %    title={On the maximum running time in graph bootstrap percolation},
            %    journal={Electronic Journal of Combinatorics},
            %    volume={24},
            %    date={2017},
            %    number={2},
            %    pages={Paper No. 2.16, 20},
            % }

            \bib{BPRS:17}{article}{
				title={On the maximum running time in graph bootstrap percolation},
				author={Bollob\'as, B\'ela},
               author={Przykucki, Micha\l},
               author={Riordan, Oliver},
               author={Sahasrabudhe, Julian},
				journal={Electronic Journal of Combinatorics},
				volume={24},
				number={2},
				year={2017},
				pages={P2.16},
				publisher={Electronic Journal of Combinatorics}
			}
			
			
			\bib{CLR:79}{article}{
				title={Bootstrap percolation on a {B}ethe lattice},
				author={Chalupa, John},
				author={Leath, Paul L.},  
				author={Reich, Gary R.},
				journal={Journal of Physics C: Solid State Physics},
				volume={12},
				number={1},
				pages={L31},
				year={1979},
				publisher={IOP Publishing}
			}

            \bib{EJKL:22}{article}{
				title={Long running time for hypergraph bootstrap percolation},
				journal={arXiv:2209.02015},
				author={Espuny D\'iaz, Alberto},
				author={Janzer, Barnabás},
				author={Kronenberg, Gal},
				author={Lada, Joanna},
				year={2022},
				pages={To appear in European Journal of Combinatorics},
			}
			
			
			% \bib{FMS:22}{article}{
			% 	title={Maximum running times for graph bootstrap percolation processes},
			% 	author={Fabian, David},
			% 	author={Morris, Patrick},
			% 	author={Szab{\'o}, Tibor},
			% 	journal={Discrete Mathematics Days 2022},
			% 	volume={263},
			% 	pages={134},
			% 	year={2022},
			% 	publisher={Ed. Universidad de Cantabria}
			% }
			
			\bib{FMS:23}{article}{
				title={Slow graph bootstrap percolation I: Cycles}, 
				author={Fabian, David},
				author={Morris, Patrick},
				author={Szab{\'o}, Tibor},
				year={2023},
				journal={arXiv:2308.00498}, 
			}
			
			\bib{FMS:25}{article}{
				title={Slow graph bootstrap percolation II: Accelerating properties}, 
				author={Fabian, David},
				author={Morris, Patrick},
				author={Szab{\'o}, Tibor},
				journal={Journal of Combinatorial Theory, Series B},
				volume={172},
				pages={44--79},
				year={2025},
			}
			
			% \bib{FMS:2508}{article}{
			% 	title={Slow graph bootstrap percolation III: Chain constructions}, 
			% 	author={Fabian, David},
			% 	author={Morris, Patrick},
			% 	author={Szab{\'o}, Tibor},
			% 	year={2025},
			% 	journal={arXiv:2508.03835},
			% }

            \bib{FMS:26}{article}{
                  title={Graph bootstrap percolation -- a discovery of slowness}, 
                  author={Fabian, David},
				  author={Morris, Patrick},
				  author={Szab{\'o}, Tibor},
                  year={2026},
                  journal={arXiv:2602.12736},
            }
			
			\bib{GKP:17}{article}{
				title={The time of graph bootstrap percolation},
				author={Gunderson, Karen},
				author={Koch, Sebastian},
				author={Przykucki, Micha{\l}},
				journal={Random Structures \& Algorithms},
				volume={51},
				number={1},
				pages={143--168},
				year={2017},
				publisher={Wiley Online Library}
			}

            \bib{HL:22}{article}{
				title = {The maximal running time of hypergraph bootstrap percolation},
				journal = {arXiv preprint arXiv:2208.13489},
				author = {Hartarsky, Ivailo},
				author = {Lichev, Lyuben},
				year = {2022},
				pages={To appear in SIAM Journal on Discrete Mathematics}
			}
            
			\bib{M:15}{article}{
				title={The saturation time of graph bootstrap percolation},
				author={Matzke, Kilian},
				journal={arXiv:1510.06156},
				year={2015}
			}

            \bib{M:17}{article}{
				title={Bootstrap percolation, and other automata},
				author={Morris, Robert},
				journal={European Journal of Combinatorics},
				volume={66},
				pages={250--263},
				year={2017},
				publisher={Elsevier}
			}

            \bib{M:06}{article}{
               author={Mubayi, Dhruv},
               title={A hypergraph extension of Tur\'an's theorem},
               journal={J. Combin. Theory Ser. B},
               volume={96},
               date={2006},
               number={1},
               pages={122--134},
               issn={0095-8956},
            }
			
			
			\bib{NR:22}{article}{
				title={On the {Running} {Time} of {Hypergraph} {Bootstrap} {Percolation}},
				journal = {arXiv preprint arXiv:2206.02940},
				pages = {To appear in Electronic Journal of Combinatorics},
				author = {Noel, Jonathan A.},
				author = {Ranganathan, Arjun},
				year = {2022},
				
			}
			
			
			
			
			
		\end{biblist}
	\end{bibdiv}
\end{document}